\documentclass{amsart}
\usepackage{amssymb,amsmath,amsthm}
\usepackage{amsrefs}

\begin{document}
\title[A Limit Theorem for Particle Current in SEP]{A Limit Theorem for Particle Current in the Symmetric Exclusion Process}
\author{Alexander Vandenberg-Rodes}
\address{University of California--Los Angeles, Mathematics Department, Box 951555, Los Angeles, CA 90095-1555}
\date{\today}
\email{avandenb@math.ucla.edu}
\thanks{A. Vandenberg-Rodes is partially supported by NSF grants DMS-0707226 and DMS-0301795.}
\abstract
Using the recently discovered strong negative dependence properties of the symmetric exclusion process, we derive general conditions for when the normalized current of particles between regions converges to the Gaussian distribution. The main novelty is that the results do not assume any translation invariance, and hold for most initial configurations.
\endabstract
\maketitle

\newcommand{\E}{\mathbb{E}}
\newcommand{\I}{\mathfrak{Im}}
\newcommand{\N}{\mathbb{N}}
\newcommand{\R}{\mathbb{R}}
\newcommand{\C}{\mathbb{C}}
\newcommand{\Z}{\mathbb{Z}}
\newcommand{\Var}{\operatorname{Var}}
\newcommand{\Cov}{\operatorname{Cov}}
\newtheorem{theorem}{Theorem}
\newtheorem{prop}{Proposition}
\newtheorem{lemma}{Lemma}
\section{Introduction}

The {\it exclusion process} on a countable set $S$ is a continuous-time Markov process describing the motion of a family of Markov chains on $S$, subject to the condition that each site can contain only one particle at a time. With the assumption that the jump rates from sites $x$ to sites $y$ satisfy $p(x,y)=p(y,x)$, the resulting process is termed the {\it symmetric} exclusion process (SEP). See \cite{liggettbook} for the construction and the general ergodic theory. 

In conservative particle systems such as the exclusion process and the zero-range process -- systems where particles are neither created nor destroyed -- one topic of study is the bulk flow or current of particles. By this we mean the net amount of particles that have flowed from one part of the system into the other. Finding the expected current in such systems is usually quite straightforward, however, given the interdependence of the particle motions, characterizing the current fluctuations is a harder problem. In the case of asymmetric exclusion on the integer lattice, the variance of the current as seen by a moving observer has been shown to have the curious order of $t^{2/3}$, with connections to random matrix theory \cites{sepp, spohn, tracy}. 

For symmetric exclusion on $\Z$, when only nearest-neighbor jumps are allowed, the current flow is intimately tied to the classical problem of determining the motion of a tagged particle. This is especially clear when the process is started from the equilibrium measure $\nu_\rho$ -- the homogeneous product measure on $\{0,1\}^\Z$ with density $\rho$. Since particles cannot jump over each other, and the spacings between subsequent particles are independent geometric-$(\rho)$ random variables, the tagged particle's displacement is asymptotically proportional to the current across the origin. In this case Arratia \cite{arratia} gave the first central limit theorem for a tagged particle, and Peligrad and Sethuraman \cite{seth} showed process-level convergence of the current (and hence of a tagged particle) to a fractional Brownian motion. Non-equilibrium results were obtained by Jara and Landim \cites{JaraLandim,lararandom} under a hydrodynamic rescaling of the process, even with non-translation invariant jump rates (quenched random bond disorder). The heat equation machinery used there requires the initial distributions to be smooth profiles, giving results only in an average sense.

More recently, Derrida and Gerschenfeld \cite{derrida} applied techniques used for the more difficult asymmetric exclusion \cite{tracy} to SEP, obtaining the asymptotic distribution of the (non-normalized) current. Although their results are sharp, translation invariance of the jump rates and a step-initial condition seem to be required by that approach.

Meanwhile, a general negative dependence theory with application to the symmetric exclusion process was developed by Pemantle \cite{pemantle}, and Borcea, Br\"and\'en and Liggett \cite{branden}, which had immediate application to the current when SEP is started from a deterministic initial state \cite{liggettarticle}.

In this paper we further exploit the negative dependence theory in this direction, obtaining a central limit theorem for the current throughout a wide range of transition rates and initial conditions.

\section{Particle current}
The original problem as described by Pemantle \cite{pemantle} was proved and generalized to the following by Liggett \cite{liggettarticle}. Consider SEP on $\Z$ with translation invariant transition probabilities that describe a random increment with finite variance, i.e., \[\sum_{n>0}n^2p(0,n) <\infty.\] Start with particles initially occupying the whole half lattice $\{x\in \Z;x\leq 0\}$. Then the current of particles across the origin after time $t$, \[W_t=\sum_{x>0} \eta_t(x),\] satisfies the central limit theorem \[\frac{W_t-\E W_t}{\sqrt{\Var(W_t)}} \Rightarrow \mathcal N(0,1) \ \mbox{ in distribution.}\] It was conjectured in \cite{liggettarticle} that this result would also hold in the case where the transition probabilities lie in the domain of a stable law of index $\alpha>1$. We will show this in Section 5.

In this paper we consider the following general setting:
Let $S$ be an arbitrary countable set. For a partition $S=A\cup B$, we think of the net current of particles from $A$ to $B$ to be \[W_t = W^+(t)-W^-(t),\] where $W^+(t)$ is the number of particles that start in $A$ and end up in $B$ at time $t$, and $W^-(t)$ is the number of particles that start in $B$ and end up in $A$. As the usual construction of SEP does not distinguish particles, we make this quantity rigorously defined through Harris' ``stirring'' representation, as used by De Masi and Ferrari \cite{demasiferrari}. The key to the stirring representation is to notice that in SEP, particles and holes both have the same transition rates. Hence we first define a larger process of randomly jumping labels, which are later reduced to either a particle (1), or hole (0).

At time $t=0$, we place at each site $x\in S$ the label $x$. For each unordered pair $(x,y)$ of points, we place a Poisson process (clock) $N^{x,y}$ with parameter $p(x,y)$. When the clock $N^{x,y}$ rings, the labels at $x$ and $y$ switch places. Let $\xi_t^x$ denote the position at time $t$ of the label $x$, so in particular $\xi^x_0 = x$. Under reasonable conditions on the rates, the random process $\{\xi_t^x;x\in S, t\geq 0\}$ is well defined on a set of full measure. Let $L_t(x)$ denote the label occupying site $x$ at time $t$. Given an initial condition $\eta\in \{0,1\}^S$, we set \[\eta_t(x) = \eta(L_t(x)).\] Notice that if the clock $N^{x,y}$ rings at time $t$, this produces an effect on the state $\eta_t$ if and only if there is one particle and one hole between sites $x$ and $y$; in that case they switch locations. This gives one construction of SEP. 	

Define the current from $A$ to $B$ as 
\begin{equation}\label{stirring}
W_t = \sum_{x\in A}\eta(x)1_{\{\xi_t^x\in B\}} - \sum_{x\in B}\eta(x)1_{\{\xi_t^x\in A\}}.
\end{equation}
This is well defined for any initial condition $\eta$ as long as \begin{equation}\label{current}
\E\bigg(\sum_{x\in B}1_{\{\xi_t^x\in A\}}\bigg)<\infty.
\end{equation}
When $\eta$ contains only finitely many particles, the current can be written as just \[W_t =\sum_{x\in B}\Big\{\eta_t(x)-\eta(x)\Big\}.\] This coincides with the definition given above, because
\begin{align*}
\sum_{x\in B}\Big\{\eta(L_t(x))-\eta(x)\Big\} &= \sum_{x\in B}\sum_{y\in S}\eta(y)1_{\{\xi^y_t=x\}}-\sum_{y\in B}\eta(y)\\
&=\sum_{y\in S}\eta(y)1_{\{\xi_t^y\in B\}}-\sum_{y\in B}\eta(y)\\
&=\sum_{y\in A}\eta(y)1_{\{\xi_t^y\in B\}}-\sum_{y\in B}\eta(y)[1-1_{\{\xi_t^y\in B\}}],
\end{align*} which is precisely the expression (\ref{stirring}). 

Recalling that $p(\cdot,\cdot)$ give the (symmetric) transition rates of individual particles in the exclusion process under consideration, we henceforth let $X_t$ be the one-particle Markov chain on $S$ with those transition rates.

For instance, suppose $S=\Z$ and $A=\{x\leq 0\}$, $B=\{x>0\}$. Each $\xi_t^x$ has the same distribution as $X_t$ started from the site $x$, though of course for different $x$ the Markov chains are highly dependent. Then under very mild conditions on the rates, such as \[p(x,y)\leq C|x-y|^{-\alpha-1},\ \ \alpha>1,\] we can compare (using a coupling argument) $X_t$ to a translation-invariant random walk $Z_t$, having finite first moment, to show that $P^x(X_t\leq 0)\leq P^0(Z_t\geq x)$. Condition (\ref{current}) then holds, because \[\E\bigg(\sum_{x>0}1_{\{\xi_t^x\leq 0\}}\bigg)\leq \sum_{x>0}P^0(Z_t\geq x)=\E(Z_t^+)<\infty.\]
 
We say that the partition $S=A\cup B$ is {\it balanced} if there is a $c>0$, not depending on $x$, such that \begin{equation}\label{balance}
c<\liminf_{t\rightarrow \infty} P^x(X_t\in A) \leq \limsup_{t\rightarrow\infty} P^x(X_t\in A)<1-c.
\end{equation}
Here is our main theorem: 
\begin{theorem}\label{maintheorem} Let $S=A\cup B$ be any balanced partition of $S$, and $\eta\in\{0,1\}^S$ be a (deterministic) initial condition for $\eta_t$ -- the symmetric exclusion process on $S$. Suppose (\ref{current}) holds at all times, and that \begin{equation}\label{varsum}
\sup_{t\geq 0}\E^\eta\bigg(\sum_{\eta(x)=1}(1-\eta_t(x))\bigg)=\infty.
\end{equation}
Then the current $W^\eta_t$ of particles between $A$ and $B$ satisfies the central limit theorem \[\overline{W^\eta_t} :=\frac{W^\eta_t-\E W^\eta_t}{\sqrt{\Var W^\eta_t}} \Rightarrow^d \mathcal{N}(0,1).\]

Furthermore, we have the following rate of convergence in the Levy metric: \[d(\overline{W^\eta_t},\mathcal{N})\leq C(\Var W^\eta_t)^{-\frac12}.\]
\end{theorem}

Condition (\ref{varsum}) is a measure of how rigid the system is: by varying the time parameter, the expected number of initially occupied sites that are then empty needs to be unbounded. 

The reader can skip to the last section to see these conditions checked for a couple of examples.

\section{Negative dependence and SEP}
Because of the hard-core repulsion of particles, the Symmetric Exclusion process tends to spread out more than independent particles would. 
One example of this is the following correlation inequality of Andjel \cite{andjel}: for disjoint subsets $A,B$ of $S$, and starting configuration $\eta$,
\begin{equation}\label{andjelineq}
P^\eta(\eta_t\equiv 1 \mbox{ on }A\cup B)\leq P^\eta(\eta_t\equiv 1\mbox{ on }A)P^\eta(\eta_t\equiv 1\mbox{ on }B).
\end{equation}
There is already a well-developed theory of positive correlations, with results such as the celebrated FKG inequality. There, one states that a measure $\mu$ is positively associated if for all monotone increasing functions $f,g$ -- assuming the natural partial ordering on $\{0,1\}^S$,
\begin{equation*}
\int fg d\mu\geq \int f d\mu \int g d\mu.
\end{equation*}
Many processes with spin-flip dynamics -- such as the Ising and Voter models -- are known to preserve positive association. That is, assuming an initial distribution that is positively associated, the distribution of the process at later times is still positively associated.
One may consider the following analogue for negative correlations: we say that $\mu$ is negatively associated if 
\begin{equation}\label{negcor}
\int fg d\mu\leq \int f d\mu \int g d\mu,
\end{equation}
for all increasing functions $f,g$ that depend on disjoint sets of coordinates. The latter condition is a reflection of the fact that any random variable is positively correlated with itself. Unfortunately SEP does not preserve negative correlations \cite{liggett2}, however, there is a useful subclass of measures -- introduced in \cite{branden} -- that is preserved by the evolution of SEP.

A multivariate polynomial $f\in \C[z_1,\dotsc,z_n]$ is called {\it stable} if 
\begin{equation}\label{stability}\I(z_j)>0\mbox{ for all } 1\leq j\leq n\ \Rightarrow \ f(z_1,\dotsc,z_n)\neq 0.\end{equation}

We then say that the probability measure $\mu$ on $\{0,1\}^n$ is {\it strongly Rayleigh} if its associated generating polynomial \[f_\mu(z) = \E^\mu(z_1^{\eta(1)}\cdots z_n^{\eta(n)})\] is stable.

In the setting $\{0,1\}^S$ for $S$ infinite, we say that the measure $\mu$ is strongly Rayleigh if every projection of $\mu$ onto finitely many coordinates is strongly Rayleigh. It is easy to check that product measures on $\{0,1\}^S$ are strongly Rayleigh.

Two key results (among many) were shown in \cite{branden}:
\begin{enumerate}
\item Strongly Rayleigh measures are negatively associated.
\item The evolution of SEP preserves the class of strongly Rayleigh measures.
\end{enumerate}
The distributional limits found in \cite{liggettarticle} relied upon the following result:
\begin{prop}\label{bernoulli}
Suppose $\mu$ is strongly Rayleigh and $T\subset S$. Then $\sum_{x\in T}\eta(x)$ has the same distribution as $\sum_{x\in T} \zeta_x$, for a collection $\{\zeta_x; x\in T\}$ of independent Bernoulli random variables.
\end{prop}
Combining the above proposition with standard conditions for convergence to the normal distribution yields a central limit theorem for strongly Rayleigh random variables.
\begin{prop}\label{clt}
Suppose that for each $n$ the collection of Bernoulli random variables $\{\eta_n(x);x\in S\}$ determines a strongly Rayleigh probability measure. Furthermore, assume the variances $\Var(\sum_{x\in S} \eta_n(x)) \rightarrow \infty$ as $n\rightarrow \infty$. Then \[\frac{\sum_{S}\eta_n(x)-\E(\sum_{S}\eta_n(x))}{\sqrt{\Var(\sum_{S}\eta_n(x))}}\stackrel{d}{\Rightarrow} \mathcal{N}(0,1) \mbox{ as }n\rightarrow \infty.\]
\end{prop}
These kinds of results were already known for determinantal processes -- see \cites{peres,lyons} -- although this is hardly a coincidence, as a large subset of determinantal measures are strongly Rayleigh \cite{branden}*{Proposition 3.5}.

During the writing of this paper, it became apparent that the argument given in \cite{liggettarticle} for the first proposition only holds for finite subsets $T$. While the results here and in \cite{liggettarticle} can be modified to accommodate this deficiency, in section 4 we will give a proof for the infinite case that may be of independent interest.

The proof of Theorem \ref{maintheorem} hinges upon the following proposition:
\begin{prop}\label{prop1}
Under the conditions of Theorem \ref{maintheorem}, \[\Var W^\eta_t\rightarrow \infty \mbox{ as $t\rightarrow\infty$.}\]
\end{prop}
Suppose we start with $\eta\equiv 1$ on $A$ and $\eta\equiv 0$ on $B$. It is trivial to check that product measures on $\{0,1\}^S$ are strongly Rayleigh, so by Proposition \ref{clt}, in order to show convergence to the normal distribution we need only show that the variance $\Var W^\eta_t\rightarrow \infty$ as $t\rightarrow \infty$. Then we may write the current variance as \[\Var(W^\eta_t)=\Var\bigg(\sum_{x\in B} \eta_t(x)\bigg) = \sum_{x\in B} \Var(\eta_t(x)) + \sum_{\substack{x,y\in B\\x\neq y}} \Cov(\eta_t(x),\eta_t(y)).\] Because of the negative association property of symmetric exclusion, all the off-diagonal covariances ($x \neq y$) are negative, while of course the diagonal terms are positive. The approach in \cite{liggettarticle} for the Pemantle problem described above was to compute the exact asymptotics of the diagonal terms, and then estimate the negative (off-diagonal) terms to be at most some fixed percentage smaller. Getting tight-enough bounds on the negative terms was already tricky in that case - considering even slightly more general transition functions $p(x,y)$ seems to require quite delicate analysis to obtain bounds that even approached the positive terms' asymptotics. To get around this obstacle, we do a generator computation in order to rewrite the {\it positive} variances, and obtain term-by-term domination of the off-diagonal covariances.

There is one additional complication when extending the result to the more general initial conditions of Theorem \ref{maintheorem}: when $\eta$ contains infinitely many particles in both $A$ and $B$ we cannot write $W^\eta_t$ as a convergent sum of occupation variables. Instead, we approximate by considering initial conditions with only finitely many particles. Consider first the following lemma.
\begin{lemma}\label{lem1}
Suppose $S_n\nearrow S$ is an increasing sequence of finite subsets, and for $\eta\in \{0,1\}^S$ define $\eta^{n}(x) = 1_{x\in S_n} \eta(x)$. Then for fixed $t\geq 0$ such that (\ref{current}) holds,
\[W_t^{\eta^{n}}\rightarrow W_t^{\eta} \mbox{ in }L^2 \mbox{ as }n\rightarrow\infty.\]
\end{lemma}
\begin{proof}
Using the stirring representation (\ref{stirring}) and the inequality $(a-b)^2\leq a^2+b^2$ for $a,b\geq 0$, 
\begin{align*}
\E(W^\eta_t - W^{\eta^n}_t)^2\leq \ &\E\Big(\sum_{x\in A}1_{\{\xi^x_t\in B\}}(\eta(x)-\eta^n(x))\Big)^2 \\
+ &\E\Big(\sum_{x\in B}1_{\{\xi^x_t\in A\}}(\eta(x)-\eta^n(x))\Big)^2.
\end{align*}
Both expectations are dealt with identically, so we consider here only the first one. Expanding it gives
\begin{align*}
\sum_{x,y\in A}\E 1_{\{\xi^x_t\in B\}}1_{\{\xi^y_t\in B\}}(\eta(x)-\eta^n(x))(\eta(y)-\eta^n(y))\\
\leq \Big(\E \sum_{x\in A} 1_{\{\xi^x_t\in B\}}(\eta(x)-\eta^n(x))\Big)^2 + \E\sum_{x\in A} 1_{\{\xi^x_t\in B\}}(\eta(x)-\eta^n(x)).
\end{align*}
The inequality here follows by negative dependence (\ref{andjelineq}), and the expectations appearing in the last line converge to zero by Dominated Convergence.
\end{proof}

\begin{proof}[Proof of Theorem 1.] For $(\eta,\eta^n)$ as above we note by the triangle inequality that 
\begin{equation}\label{triangle}
d(\overline{W^\eta_t},\mathcal{N})\leq d(\overline{W^\eta_t},\overline{W^{\eta^n}_t}) + d(\overline{W^{\eta^n}_t},\mathcal{N}).
\end{equation}
Now recall from Proposition \ref{bernoulli} that 
\[W^{\eta^n}_t + \sum_{x\in A}\eta^n(x) \stackrel{d}{=} \sum_{x\in S_n}\zeta^n_{t,x}\] where the $\zeta$ are Bernoulli and independent in $x$ for each $n$ and $t$. Normalize both sides to see that \[\overline{W_t^{\eta^n}}\stackrel{d}{=} \overline{\sum_{x\in S_n}\zeta^n_{t,x}}.\] Hence by Esseen's inequality \cite{petrov}*{V, Theorem 3}, 
\begin{align*}
d(\overline{W^{\eta^n}_t},\mathcal{N})\leq C\left[\sum_{x\in S_n} \Var (\zeta^n_{t,x})\right]^{-\frac32}\left[\sum_{x\in S_n} \E|\zeta^n_{t,x}-\E\zeta^n_{t,x}|^3\right]\\
\leq C\left[\sum_{x\in S_n} \Var (\zeta^n_{t,x})\right]^{-\frac12},
\end{align*}
because $\zeta$ Bernoulli implies that $\E|\zeta-\E\zeta|^3\leq \Var(\zeta)$ by an easy calculation.
Taking $n\rightarrow\infty$ above and in (\ref{triangle}), and using Lemma \ref{lem1}, 
\[d(\overline{W^\eta_t},\mathcal{N})\leq C[\Var(W^\eta_t)]^{-\frac12}.\]
Applying Proposition \ref{prop1} finishes the proof.
\end{proof}

Our proof of Proposition \ref{prop1} relies upon the following representation for the on-diagonal covariances.
\begin{lemma}\label{varlemma}
Let $X_t$ be defined as above Theorem \ref{maintheorem}. Then for any $\eta\in \{0,1\}^S$ with finite support (i.e., $\eta(x)=1$ for only finitely many $x$),
\begin{equation}\label{vareq}
\sum_{x\in S} \Var(\eta_t(x)) = \int_0^t\sum_{\substack{x\neq y\\x,y\in S}}p(x,y)[\E^y\eta(X_s)-\E^x\eta(X_s)]^2ds.
\end{equation}
\end{lemma}
\begin{proof}
(Essentially a generator computation). From the duality theory of SEP \cite{liggettbook}*{VIII, Theorem 1.1}, $\E^\eta\eta_s(x)=\E^x\eta(X_s)$. So we can write \[\sum_{x\in S} \Var \eta_s(x) =\sum_{x\in S} \Big\{\E^\eta \eta_s(x) - (\E^\eta\eta_s(x))^2 \Big\}= \sum_{x\in S} \Big\{\E^x \eta(X_s)-[\E^x \eta(X_s)]^2\Big\}.\] Let $U$ and $\{U(t);t\geq 0\}$ be the generator and semi-group for $X_t$. Changing into the language of generators, we have \[\sum_{x\in S}\Var(\eta_s(x)) =\sum_{x\in S} \Big\{U(s)\eta(x)-[U(s)\eta(x)]^2\Big\}.\] Now take the derivative w.r.t. $s$. For the second equality below, recall that \[Uf(x) = \sum_{y\in S:y\neq x}p(x,y)[f(y)-f(x)],\] for bounded $f$.
\begin{align}\label{reverse}
	\frac d{ds}  \sum_{x\in S} \Big\{U(s)\eta(x)&-[U(s)\eta(x)]^2\Big\} = \sum_{x\in S} U[U(s)\eta](x)[1-2U(s)\eta(x)]\notag\\
	=\sum_{\substack{x\neq y\\x,y\in S}} p(x,y)&[\E^y \eta(X_s)-\E^x \eta(X_s)][1-2\E^x\eta(X_s)]\notag\\
	=\sum_{x\neq y} p(x,y)&[\E^y\eta(X_s)-\E^x \eta(X_s)]^2\notag\\
	+ &\sum_{x\neq y} p(x,y)[\E^y\eta(X_s)-\E^x\eta(X_s)][1-\E^y\eta(X_s)-\E^x\eta(X_s)],
\end{align}
where all sums converge absolutely because $\eta$ has finite support. Since $p(x,y)=p(y,x)$, exchanging $x$ and $y$ in the latter sum in (\ref{reverse}) shows it to be its own negative, hence zero. We thus conclude that \[\frac{d}{ds}\sum_{x\in S}\Var(\eta_s(x))=\sum_{\substack{x\neq y\\x,y\in S}}p(x,y)[\E^y\eta(X_s)-\E^x \eta(X_s)]^2.\]Integrating from $0$ to $t$ finishes the proof.
\end{proof}

\begin{proof}[Proof of Proposition \ref{prop1}]
Considering only initial conditions $\eta\in \{0,1\}^S$ containing finitely many particles, the net current from $A$ to $B$ can be written as 
\[W^\eta_t =\sum_{x\in B} \Big\{\eta_t(x)-\eta(x)\Big\} = \sum_{x\in A} \Big\{\eta(x)-\eta_t(x)\Big\}.\]
We symmetrize this expression:
\[2W^\eta_t = \sum_{x\in S} [H(x)\eta_t(x)-H(x)\eta(x)], \mbox{ where } H(x) = \left\{ \begin{array}{ll} 1 &\mbox{ if $x\in B$}\\ -1 &\mbox{ if $x\in A$,}\end{array}\right. \]
and consider the variance: 
\begin{align*}
4\Var (W^\eta_t) = \sum_{x\in S}\Var(\eta_t(x)) + \sum_{\substack{x\neq y\\x,y\in S}} H(x)H(y)\Cov(\eta_t(x),\eta_t(y)).
\end{align*}

We first deal with the covariances above, proceeding almost identically to \cite{liggettarticle}. Let $\{U_2(t);t\geq 0\}$  be the semigroup for two identical, independent Markov chains with symmetric kernel $p(x,y)$, and let $U_2$ be its infinitesimal generator. Specifically, 
\begin{equation}\label{U} U_2f(x,y) = \sum_{z\in S} \Big\{p(x,z)[f(z,y)-f(x,y)] + p(y,z)[f(x,z)-f(x,y)]\Big\}.\end{equation}
Let $\{V_2(t);t\geq 0\}$ and $V_2$ be the semigroup and generator for that process with the exclusion interaction, i.e.
\begin{equation}\label{V} V_2f(x,y) = \sum_{z\neq y} p(x,z)[f(z,y)-f(x,y)] + \sum_{z\neq x} p(y,z)[f(x,z)-f(x,y)].\end{equation} 

By a slight abuse of notation, define \[\eta(x,y) = \eta(x)\eta(y) \ \mbox{ and } H(x,y)=H(x)H(y).\] By duality and the integration by parts formula,
\begin{align}\label{delta}
-\sum_{x\neq y}& H(x,y)\Cov(\eta_t(x),\eta_t(y)) = \sum_{x\neq y} H(x,y)[U_2(t)-V_2(t)]\eta(x,y)\notag\\
=&\int_0^t\sum_{x\neq y}H(x,y) V_2(t-s)[U_2-V_2]U_2(s)\eta(x,y)ds.
\end{align}

Now from (\ref{U}) and (\ref{V}) we have that
\begin{align*}
 [U_2-V_2]U_2(s)\eta(x,y) &= p(x,y)\Big\{U_2(s)\eta(x,x)+U_2(s)\eta(y,y)-2U_2(s)\eta(x,y)\Big\}\\
&=p(x,y)[\E^y\eta(X_s)-\E^x\eta(X_s)]^2,
\end{align*}
which we substitute into (\ref{delta}), also using the fact that $V_2(t-s)$ is a symmetric linear operator on the space of functions on $\{(x,y)\in S^2;x\neq y\}$:
\begin{align*}
	-\sum_{x\neq y} &H(x)H(y)\Cov (\eta_t(x),\eta_t(y))\notag\\
	 &= \int_0^t \sum_{x\neq y} p(x,y)[\E^y\eta(X_s)-\E^x\eta(X_s)]^2 V_2(t-s)H(x,y)ds\notag\\
	&\leq \int_0^t \sum_{x\neq y} p(x,y)[\E^y\eta(X_s)-\E^x\eta(X_s)]^2 U_2(t-s)H(x,y)ds,
\end{align*}
by a standard inequality comparing interacting and non-interacting particles \cite{liggettbook}*{VIII, Proposition 1.7}.

Combining Lemma \ref{varlemma} and the above estimate we obtain:
\begin{align}\label{together}
		4\Var (W^\eta_t)\geq \int_0^t\sum_{x\neq y} p(x,y)[\E^y\eta(X_s)-\E^x\eta(X_s)]^2q_{t-s}(x,y)ds,
\end{align}
where \[q_s(x,y) =1-U_2(s)H(x,y)= 1-[1-2P^x(X_s\in A)][1-2P^y(X_s\in A)].\] Notice that there is a constant $c'>0$, depending only on the $c$ in (\ref{balance}), such that $q_s(x,y)>c'$ for each $x,y\in S$ and then $s$ large enough. So for fixed $T>0$, applying Fatou's Lemma twice, then using Lemma \ref{varlemma} again,
\begin{align*}
4\liminf_{t\rightarrow\infty}\Var(W^\eta_t)\geq &\int_0^T\sum_{x\neq y} p(x,y)[\E^y\eta(X_s)-\E^x\eta(X_s)]^2\liminf_{t\rightarrow\infty}q_{t-s}(x,y)ds\\
\geq &\int_0^T  \sum_{x\neq y}p(x,y)[\E^y\eta(X_s)-\E^x\eta(X_s)]^2c' ds\\
=&c'\sum_{x\in S}\Var(\eta_T(x)).
\end{align*}

Now for any finite $S'\subset S$,
\begin{align*}
\sum_{x\in S'} \Var(\eta_T(x)) = \sum_{x\in S'}\sum_{y\in S}\eta(y)p_T(x,y)\bigg[1-\sum_{z\in S}\eta(z)p_T(x,z)\bigg]\\
=\sum_{\eta(y)=1}\sum_{\eta(z)=0}\sum_{x\in S'}p_T(y,x)p_T(x,z)\rightarrow\sum_{\eta(y)=1}\sum_{\eta(z)=0}p_{2T}(z,y),
\end{align*}
as $S'\nearrow S$, by monotone convergence. By duality again, this is precisely \[\sum_{\eta(y)=1}\E^\eta(1-\eta_{2T}(y)).\] Hence by (\ref{varsum}), \[\liminf_{t\rightarrow \infty} \Var(W^\eta_t)=\infty,\] as desired.
\end{proof}

\section{Sums of strongly Rayleigh random variables}
In this section we prove Proposition 1 for infinite-dimensional $\mu$. By a generalization of the Borel-Cantelli lemmas \cite{erdos}, applied to the negatively dependent events $\{\eta(x)=1\}$, \[\sum_{x\in T}\eta(x) = \infty, \ \mu \mbox{-a.s. if }\E^\mu\sum_{x\in T}\eta(x) = \infty.\] In this case the proposition is trivially true, hence we may assume that \[\E^\mu\sum_{x\in T}\eta(x)<\infty.\]
First take an increasing sequence of finite subsets $T_n\nearrow T$. Now define for $z\in \C$ the polynomials \[Q_n(z) = \E^\mu z^{\sum_{x\in T_n}\eta(x)}.\] The limit \[Q(z) = \E^\mu z^{\sum_{x\in T}\eta(x)}\] exists, and in fact $Q_n\rightarrow Q$ uniformly on compact sets. Indeed, 
\begin{align*}
|Q_n(z)-Q(z)|\leq \E^\mu\bigg|z^{\sum_{x\in T_n}\eta(x)}\Big[1-z^{\sum_{x\in T\setminus T_n}\eta(x)}\Big]\bigg|\\
\leq \E^\mu\bigg(\max\Big\{1,|z|^{\sum_{x\in T}\eta(x)}\Big\}1_{\{\eta\not\equiv 0 \mbox{ on }T\setminus T_n\}}\bigg).
\end{align*}
Now for $|z|=r>1$, $r^t$ is an increasing function of $t$, so the negative dependence property of $\mu$ (\ref{negcor}), implies that \begin{equation}\label{Q}
\E^\mu r^{\sum_{x\in T}\eta(x)}\leq\prod_{x\in T}\E^\mu r^{\eta(x)}=\prod_{x\in T}[1+(r-1)\E^\mu\eta(x)]\leq e^{r\E^\mu\sum_{x\in T}\eta(x)}<\infty.
\end{equation}
(The last inequality uses the estimate $1+x\leq e^x$.) Dominated convergence then gives the normal convergence $Q_n\rightarrow Q$. In particular, $Q$ is entire.

By stability (\ref{stability}), $Q_n(z)$ has only real zeros, furthermore, the zeros must all be negative because the coefficients of $Q_n$ are all non-negative. By classical theorems on entire functions \cite{levin}*{VIII, Theorem 1}, the limit $Q(z)$ has the form
\[Q(z)=Ce^{-\sigma z}\prod_{k=1}^\infty\bigg[1-\frac{z}{a_k}\bigg],\] for some $\sigma\leq 0$, and $a_k<0$ with $\sum|a_k|^{-1}<\infty$. It is enough to show that $\sigma=0$, because, as $Q(1)=1$ we can solve for $C$ to obtain \[Q(z) = \prod_{k=1}^\infty\frac{a_k-z}{a_k-1} = \prod_{k=1}^\infty[p_kz + (1-p_k)],\] where we set $p_k = 1/(1-a_k)$. But this last expression is just the generating function for the sum of independent Bernoulli r.v.'s having the parameters $p_k$. 

To obtain $\sigma=0$, we show that $|Q(z)|\leq e^{c|z|}$ for any $c>0$ and $|z|$ large enough. It is clear that $|Q(z)|\leq Q(|z|)$, hence we consider only $z=r>1$. Recall from (\ref{Q}) that \[Q(r)\leq\prod_{x\in T}[1+(r-1)\E^\mu\eta(x)].\] Let \[a_x=\E^\mu\eta(x).\] With $a = \sum_{x\in T}a_x<\infty$, note that $\#\{x:a_x>r^{-1/2}\}\leq ar^{1/2}$. Then by a trivial bound we have 
\begin{align*}
Q(r)\leq \bigg(\prod_{a_x>r^{-1/2}} r\bigg)\bigg(\prod_{a_x\leq r^{-1/2}} e^{(r-1)a_x}\bigg)\\
\leq r^{ar^{1/2}}\exp\bigg((r-1)\sum_{a_x\leq r^{-1/2}}a_x\bigg).
\end{align*}
As $r\rightarrow\infty$ the sum inside the exponential goes to zero, which concludes the proof.

\section{Examples}
1. Consider $S=\Z$, partitioned into $A=\{x\leq 0\}$ and $B=\{x>0\}$, with translation invariant rates $p(0,x)$ in the domain of a symmetric stable law with index $\alpha>1$. That is, \[\sum_{y\geq x}p(0,y)\sim L(x)x^{-\alpha},\ x>0,\]for a slowly varying function $L$. Consider the step initial condition $\eta$ with particles at all $x\leq 0$.
The balance condition holds by the central limit theorem for random variables in the domain of attraction of a stable law. By duality and translation invariance, \[\sum_{x\leq 0} P^\eta(\eta_t(x)=0) = \sum_{x\leq 0}\sum_{y>0}P^x(X_t=y) = \sum_{n>0}nP^0(X_t=n) = \E^0 X_t^+\rightarrow \infty.\] (In fact, it grows at rate $t^{1/\alpha}$). This shows (\ref{varsum}), and the above expression is the same as $\E^\eta W_t$, so all conditions of Theorem \ref{maintheorem} have been verified.

2. Now consider the one-dimensional exclusion process in an random environment, with the same partition as in the previous example. The random environment is described by $\{\omega_i\}$, an iid family of random variables with $\omega\in (0,1]$ almost surely and $\E\frac{1}{\omega_i}<\infty$. For each realization $(\dotsc, \omega_{-1},\omega_0,\omega_1,\dotsc)$, we consider the exclusion process with the rates $p(i,i+1) =p(i+1,i) = \omega_i$. By the result of Kawazu and Kesten \cite{kesten}, we know that the process $\{X_{n^2t}/n\}$ converges weakly to a scaled Brownian motion, from which follows the balance condition. By the remarks after (\ref{current}) we know that the current has finite expectation for any initial placement of particles. Let us consider the case where we pick a realization $\eta$ of the homogeneous product measure $\nu_\rho$. Then almost surely-$\nu_\rho$ there are infinitely many $x\in \Z$ such that $\eta(x)=1$ and $\eta(x+1)=0$, so from the ergodicity of the environment, \[\infty = \sum_{\substack{\eta(x)=1\\ \eta(x+1)=0}}P^x(X_t=x+1)\leq \sum_{\eta(x)=1}P^x(\eta(X_t)=0),\] which shows (\ref{varsum}) by duality.

\subsection*{Acknowledgments.} This article is part of the authors' thesis under T. M. Liggett, whom the author would like to thank for his advice and encouragement.

\end{document}